\documentclass[reqno,twoside,11pt]{amsart}

\usepackage{amsmath,amsfonts,calrsfs,fullpage,amssymb,color,verbatim,eucal,yfonts,mathrsfs}

\newtheorem{Theorem}{Theorem}[section]
\newtheorem{Definition}[Theorem]{Definition}
\newtheorem{Proposition}[Theorem]{Proposition}
\newtheorem{Lemma}[Theorem]{Lemma}
\newtheorem{Corollary}[Theorem]{Corollary}
\newtheorem{Remark}[Theorem]{Remark}

\makeatletter
\@addtoreset{equation}{section}

\makeatother

\def\Q{\mathbb Q}
\def\R{\mathbb R}
\def\N{\mathbb N}

\def\E{\mathbb E}

\def\eps{\varepsilon}
\def\ds{\displaystyle}
\def\m{\textfrak{m}}
\newcommand{\esssup}{\operatorname{ess\,sup}}
\newcommand{\essinf}{\operatorname{ess\,inf}}

\newcommand{\one}{1\mkern -4mu\mathrm{l}}

\title[Surface measures in infinite dimension]{\bf Surface measures in infinite dimension}

\author[G. Da Prato]{Giuseppe Da Prato}
\address{Scuola Normale Superiore\\
Piazza dei Cavalieri, 7\\ 
56126 Pisa, Italy}
\email{g.daprato@sns.it}

\author[A. Lunardi]{Alessandra Lunardi}
\address{
Dipartimento di Matematica e Informatica\\
Universit\`a di Parma\\
Parco Area delle Scienze, 53/A\\
43124 Parma, Italy}
\email{alessandra.lunardi@unipr.it}

\author[L. Tubaro]{Luciano Tubaro}
\address{
Dipartimento di Matematica\\
Universit\`a di Trento\\
Via Sommarive 14\\
38123 Povo, Italy}
\email{tubaro@science.unitn.it}

\subjclass[2010]{28C20}

\keywords{Infinite dimensional analysis, surface measures,  Gaussian measures}

\begin{document}

 \begin{abstract}  
 We construct surface measures associated to Gaussian measures in separable Banach spaces, and we prove several properties including an integration by parts formula. 
 \end{abstract}

\maketitle

\section{Introduction} 
 
Let $X$ be a separable Banach space with norm $\|\cdot\|$,  endowed with a nondegenerate centered Gaussian measure $\mu$, with covariance $Q$ and  associated Cameron--Martin space $H$.

We will construct surface measures, defined on level sets $ \{x\in X: \;G(x) = r\}$ for suitable $G:X\mapsto \R$, and prove several properties including an integration by parts formula for Sobolev functions, that involves a surface integral.
 
Surface measures in Banach spaces are not a novelty. The first steps were made in the case of Hilbert spaces, for instance in the book \cite{Sk74} where a class of smooth  surfaces was considered. 
To our knowledge, the earliest results in Banach spaces are due to Uglanov  \cite{Ug}, about surface measures on (unions of) graphs of smooth functions,  and Hertle \cite{Hertle}, that deals only with hyperplanes and spherical surfaces. 

The first systematic treatment for a more general class of surfaces was done by Airault and Malliavin in
\cite{AM}, that concerns level sets of  functions $G\in \cap_{k\in \N, p\geq 1} W^{k,p}(X,\mu)$ satisfying a  nondegeneracy condition, where $X$ is the classical Wiener space.  

The approach of \cite{AM} is through the study of the densities of suitable image measures. The same approach was considered in the books \cite{Boga,Malliavin}  in  more general contexts, and in \cite{Kuo} in the special case where $X$ is the space of the tempered distributions in $\R$. 
  The aim of this paper is to give an alternative simpler and clearer construction and description of surface measures through the image measures approach, under less regularity assumptions on $G$ with respect to \cite{AM,Boga,Malliavin}. 

A completely different approach was followed by Feyel and de La Pradelle in \cite{FP}, who defined a Hausdorff--Gauss  measure $\rho$ on the Borel sets of $X$ by finite dimensional approximations. Another very general surface measure is the perimeter measure, defined as the variation measure of the characteristic function of $\{x\in X:\;G(x)<r \}$ in the case that such characteristic function is of bounded variation, see \cite{FH} and the following papers \cite{AMMP,Hino}. It is known that under some regularity assumption on $G$, the perimeter measure coincides with $\rho$ on the level surfaces of $G$, and they coincide with the   Airault--Malliavin surface measure under further assumptions (\cite{Tracce,CLMN}). 
 
After the construction of the surface measures $\sigma_r^G$, we   show several properties of them, under minimal assumptions: they are non trivial (namely, $\sigma_r^G(X)>0$) for every $r\in (\essinf G, \esssup G)$, the support of $\sigma_r^G$ is contained in $G^{-1}( r)$, Borel sets with null Gaussian capacity are negligible with respect to $\sigma_r^G$ for every $r$, and the integration by parts formula
$$\int_{G^{-1}(-\infty, r)} (D_k\varphi -  \hat{v}_k\varphi ) \,d\mu =   \int _{G^{-1}(  r)}  \varphi   D_kG\, d\sigma_r^G, \quad k\in \N, $$
holds for  functions $ \varphi \in C^1_b(X; \R)$. Here we use standard notation: we fix any orthonormal basis $\{v_k:\;k\in \N\}$ of $H$ contained in $Q(X^*)$, $D_k \varphi$ denotes the  derivative of $\varphi $ along $v_k$, and $\hat{v}_k= Q^{-1}v_k \in X^*$. The integration by parts formula holds also for Sobolev functions, provided $\varphi   D_kG$ in the surface integral is meant in the sense of traces. Indeed, traces of Sobolev functions on the level hypersurfaces $G^{-1}( r)$ are readily defined through this approach. 

At the end of the paper we show that, under suitable assumptions,  the measures constructed here coincide with  weighted Hausdorff--Gauss surface measures, namely  for every $r\in \R$ and for every Borel set $B\subset X$ we have
$$\sigma_r^G (B) = \int_{B\cap G^{-1}( r)}\frac{1 }{|D_HG|_H}\,d\rho, $$
where $\rho$ is the above mentioned measure of \cite{FP},  $D_HG$ is the generalized gradient of $G$ along $H$, and $|\cdot |_H$ is the $H$-norm, see Sect. \ref{Notation}. This formula clarifies the dependence of $\sigma_r^G$ on $G$. 
Moreover, more refined properties of surface integrals and traces of Sobolev functions are consequences of the results of \cite{Tracce}. Also, the examples contained in \cite{Tracce} serve as examples here.    

\section{Notation and preliminaries}
\label{Notation}

We denote by $X^*$ the dual space of $X$, and by $Q:X^*\mapsto X$ the covariance of $\mu$. The Cameron--Martin space is denoted by $H$, its scalar product by $\langle \cdot , \cdot \rangle_H$ and its norm by $|\cdot |_H$. The closed ball in $H$ centered at $h_0$ with radius $r$ is denoted by $B_H(h_0, r)$. 

We fix once and for all an orthonormal basis ${\mathcal V} =\{v_k:\;k\in \N\}$ of $H$, contained in $Q(X^*)$. For every $k\in \N$ we set $\hat{v}_k = Q^{-1}(v_k)$. 

We recall that if $X$ is a Hilbert space and $X^*$ is canonically identified with $X$, then $Q$ is a compact self--adjoint operator with finite trace, and we can choose a basis $\{e_k:\;k\in \N\}$ of $X$  consisting of  eigenvectors of $Q$, $Qe_k= \lambda_ke_k$.  The space $H$ is just $Q^{1/2}(X)$ with the scalar product $\langle h_1, h_2\rangle_H= \langle Q^{-1/2}h_1,  Q^{-1/2}h_2\rangle_X$,  the set $\{v_k:= \sqrt{\lambda_k}e_k:\;k\in \N\}$ is an  orthonormal basis of $H$ and we have $\hat{v}_k(x) = \langle x, v_k\rangle_X/\lambda_k$ for each $k\in \N$ and $x\in X$. 
  
Let us recall the definition of Gaussian Sobolev spaces $W^{k,p}(X, \mu)$ for $k=1,2$, $p\geq 1$.

We say that a function $f:X\mapsto \R$ is $H$-differentiable at $x$ if there is $v\in H$ such that $f(x+h)-f(x) = \langle v, h\rangle_H + o(|h|_H)$, for every $h\in H$. In this case $v$ is unique, and we set $D_Hf(x) := v$. Moreover for every $k\in \N$ the directional derivative
$D_kf(x) : = \lim_{t\to 0}(f(x+tv_k) - f(x))/t$  exists and coincides with $\langle D_Hf(x), v_k\rangle_H$. 
It is easy to see that if $f $ is Fr\'echet differentiable at $x$ (as a function from $X$ to $\R$), then it is $H$-differentiable. In particular, the smooth cylindrical functions, namely the functions of the type $f(x) = \varphi(l_1(x), \ldots, l_n(x))$, for some $ \varphi\in C^{\infty}_b(\R^n)$, $l_1, \ldots, l_n\in X^*$, $n\in \N$, are $H$-differentiable at each $x$. 
 
$W^{1,p}(X, \mu)$ and $W^{2,p}(X, \mu)$ are the completions of the smooth cylindrical functions  in the norms 
$$\begin{array}{l}
\ds \|f\|_{W^{1,p}(X, \mu)} := \|f\|_{L^p(X, \mu)} + \bigg( \int_X \bigg(\sum_{k=1}^{\infty}  (D_{k}f(x))^2\bigg)^{p/2}\mu(dx)\bigg)^{1/p} 
\\
\\
\ds =  \|f\|_{L^p(X, \mu)} + \bigg( \int_X |D_Hf(x)|_H^{p}\mu(dx)\bigg)^{1/p}, \end{array}$$
$$\|f\|_{W^{2,p}(X, \mu)} := \|f\|_{W^{1,p}(X, \mu)} + \bigg( \int_X \bigg(\sum_{h, k=1}^{\infty}   (D_{hk}f(x))^2\bigg)^{p/2}\mu(dx)\bigg)^{1/p}. $$
Such spaces are in fact identified with subspaces of $L^p(X, \mu)$, since if $(f_n)$ and $(g_n)$ are Cauchy sequences in the norm of $W^{1,p}(X, \mu)$ (respectively, in the norm of $W^{2,p}(X, \mu)$), and converge to $f$ in $L^p(X, \mu)$, then the sequences $(D_Hf_n)$, $(D_hg_n)$ (respectively,  $(D^2_Hf_n)$, $(D^2_Hg_n)$) have equal limits in $L^p(X, \mu;H)$ (respectively, in $L^p(X, \mu;{\mathcal H}_2)$, where ${\mathcal H}_2$ is the set of all Hilbert-Schmidt bilinear forms in $H$). In other words, the operators $D_H$ and $D^2_H$, defined in the set of the smooth cylindrical functions with values in $L^p(X, \mu; H)$ and in $L^p(X, \mu;{\mathcal H}_2)$, are closable in $L^p(X, \mu)$. We still denote by $D_H$ and $D^2_H$ their closures, that are called $H$-gradient and $H$-Hessian.  

The spaces $W^{1,p}(X, \mu; H)$ are defined similarly, using $H$-valued,  instead of real valued,  cylindrical functions. The latter are the elements of  the linear span of functions such as $\psi(x) = \varphi(x)h$, with any smooth cylindrical $\varphi:X\mapsto \R$ and $h\in H$. 

The Gaussian divergence $div_{\mu}$ is defined in $W^{1,p}(X, \mu; H)$ by
$$div_{\mu} \Psi (x) = \sum_{k=1}^{\infty} (D_k \psi_k - \hat{v}_k\psi_k),  $$
where $\psi_k(x) = \langle \Psi(x),  v_k\rangle_H$, and the series converges in $L^{p}(X, \mu)$. See \cite[Prop. 5.8.8]{Boga}. 
It coincides with  (minus) the formal adjoint  of the $H$-gradient, since  we have the integration by parts formula
$$\int_X \langle D_H\varphi, \Psi \rangle_H\,d\mu = -\int_{X} \varphi \,div_{\mu} \Psi\,d\mu, \quad \varphi\in W^{1,p'}(X, \mu), \; \Psi\in W^{1,p}(X, \mu; H), $$
with $p'= p/(p-1)$.  
Taking in particular $\Psi (x)= v_k$ (constant) for any $k\in \N$, we obtain $div_{\mu} \Psi = Q^{-1}v_k  = \hat{v}_k $, and the integration formula reads as 
\begin{equation}
\label{partiX} 
\int_{X} D_k\varphi \,d\mu =   \int_{X}\hat{v}_k\varphi \,d\mu , \quad k\in \N . 
\end{equation}
We refer to \cite[Ch.5]{Boga} for equivalent definitions and properties.

The surfaces taken into consideration are level sets $  \{x\in X:\;G(x) = r\}$ of a sufficiently regular function $G$.  Namely, our $G:X\mapsto \R$ is a $C_p$-quasicontinuous function that satisfies
\begin{equation}
\label{G}
\frac{D_HG}{|D_HG|_H^2} \in W^{1,p}(X, \mu; H). 
\end{equation}
Let us recall that a function $G:X\mapsto \R$ is $C_p$-quasicontinuous if for every $\varepsilon >0$ there exists  an open set $A_{\varepsilon}$ with Gaussian capacity $C_p(A_{\eps}) <\varepsilon$, such that $G$ is continuous at any $x\notin A_{\varepsilon}$. The Gaussian capacity of an open set $A\subset X$ is defined as 
$$C_p(A):= \inf\{ \| g\|_{W^{1,p}(X, \mu)} : \; g\geq 1 \;\mu-a.e.\;{\rm on}\;A, \, g\in W^{1,p}(X, \mu)\}.$$
We recall that every element of $W^{1,p}(X, \mu)$ has a $C_p$-quasicontinuous version (\cite[Thm. 5.9.6]{Boga}). 

In addition to the Sobolev spaces, we shall consider the space $BUC(X; \R)$ of the uniformly continuous and bounded functions from $X$ to $\R$, endowed with the sup norm $\|\cdot\|_{\infty}$, and the space $C^1_b(X; \R)$ of the bounded continuously Fr\'echet differentiable functions with bounded derivative.

For any Borel function $f:X\mapsto \R$ we denote by $\mu\circ f^{-1}$ the image measure on the Borel sets of $\R$ defined by $(\mu\circ f^{-1})(I) = \mu(f^{-1}(I))$. More generally, if $\varphi:X\mapsto \R$ is another Borel function in $L^1(X, \mu)$, we define the signed measure $(\varphi \mu\circ f^{-1})(I) := \int_{f^{-1}(I)}\varphi \,d\mu $ on the Borel sets $I$ of $\R$.

 
\section{Construction of surface measures}
\label{Construction}

Throughout the paper, $G:X\mapsto \R$ is a fixed  version of an element of $W^{1,p}(X, \mu)$ (still denoted by $G$), that satisfies \eqref{G}. As pointed out in \cite{Nualart}, for $p=2$ 
an easy sufficient condition for $G$ to satisfy \eqref{G} is
$$G\in W^{2,4}(X, \mu), \quad \frac{1}{|D_HG|_H}\in L^{8}(X, \mu). $$
which is generalized to 
$$G\in W^{2,s}(X, \mu), \quad \frac{1}{|D_HG|_H}\in L^{q}(X, \mu), \quad \frac{1}{s} + \frac{2}{q} \leq \frac{1}{p}$$
if $p $ is any number $>1$. 

In the following, we shall make some further summability assumptions on the derivatives of $G$. All of them are satisfied if $G$ fulfills next condition \eqref{FP}.


\subsection{Continuity of densities}

The starting point is the following well known lemma. See \cite[Proposition 2.1.1]{Nualart}. 

\begin{Lemma}
\label{Nualart}
 $ \mu \circ G^{-1}$   is absolutely continuous with respect to  the Lebesgue measure $dr$ in $\R$. Moreover the density $q_1:=\frac{d(\mu \circ G^{-1})}{dr}$  is  given by
\begin{equation}
\label{q1}
q_1( r)=\int_{G^{-1}(-\infty, r)} div_{\mu}\bigg( \frac{D_HG}{|D_HG|_H^2} \bigg)\,d\mu,\quad r \in \R, 
\end{equation}
and it is continuous and bounded.  
\end{Lemma}

 Lemma \ref{Nualart} implies that every level surface $G^{-1}( r)$ is negligible, and that for every $\varphi\in L^1(X, \mu)$ the 
 signed measure $ \varphi\mu \circ G^{-1}$   is absolutely continuous with respect to the Lebesgue measure. In the following we shall need some properties of the density $q_{\varphi}$ of $ \varphi\mu \circ G^{-1}$ when $\varphi$ belongs to a Sobolev space. They are provided by the following lemma, which is a generalization of Lemma \ref{Nualart}. 
 
\begin{Lemma}
\label{Nualartvarphi}
Let $\varphi \in W^{1,p'}(X, \mu)$.   Then   $ \varphi\mu \circ G^{-1}$   is absolutely continuous with respect to the Lebesgue measure, and  the  corresponding density $q_\varphi $  is    given by
\begin{equation}
\label{e4m}
q_\varphi( r)=\int_{G^{-1}(-\infty, r)}\bigg( \varphi \,div_{\mu}\bigg( \frac{D_HG}{|D_HG|_H^2} \bigg)+\langle  D_H\varphi,\frac{D_HG}{|D_HG|_H^2} \rangle_H\bigg)\,d\mu,
\end{equation}
and it is continuous and bounded. There is $C>0$ independent of $\varphi$ such that 
\begin{equation}
\label{e5m}
|q_\varphi( r)| \leq C\| \varphi \|_{W^{1,p'}(X, \mu)},\quad r\in \R . 
\end{equation}
 \end{Lemma} 
 \begin{proof}
Fix  $ \alpha <\beta\in \R$ and set
 $$ f( r):=\one_{[\alpha,\beta]}( r),\quad h( r)=\int_{-\infty}^r f(s)ds =\left\{ \begin{array}{ll} 0 & {\rm if}\;\;r\leq \alpha, \\
 r-\alpha & {\rm if}\; \;\alpha \leq r \leq \beta, \\
 \beta-\alpha & {\rm if}\; \;r\geq \beta
 \end{array}\right. $$
 Since $h$ is Lipschitz continuous, then $h\circ G\in W^{1,p}(X,\mu)$  and
 $$
D_H (h\circ G)=(f\circ G) D_HG. 
 $$
 Therefore
 $$
 f\circ G=\one_{[\alpha,\beta]}\circ G =\frac{\langle D_H(h\circ G),D_HG \rangle_H}{|D_HG|_H^2}
 $$
 Assume that $\varphi \in C^1_b(X)$. Multiplying by $\varphi$ and integrating both sides  yields
 $$
  \begin{array}{lll}
\ds \int_{G^{-1}(\alpha, \beta)}  \varphi \,d\mu &=&\ds\int_X   \varphi\,\frac{\langle D_H(h\circ G),D_HG \rangle_H}{|D_HG|_H^2}\,d\mu\\
 \\
 &=&- \ds\int_X (h\circ G)\,div_{\mu}\bigg( \varphi\, \frac{\langle D_H(h\circ G),D_HG \rangle_H}{|D_HG|_H^2}\bigg)\,d\mu.
  \end{array}
 $$
 Set as before $\psi=D_HG\,|D_HG|_H^{-2}$. Then,    $div_{\mu}(\varphi\,\psi)=\varphi\,div_{\mu}(\psi)+\langle D_H\varphi,\psi\rangle_H$. 
 Therefore, 
 $$\begin{array}{lll}
\ds \int_{G^{-1}(\alpha, \beta)}  \varphi \,d\mu & = & -\ds \int_X (h\circ G)\,div_{\mu}( \varphi\,\psi)\,d\mu = -\int_X \bigg(\int_{\R}\one_{G^{-1}(r, +\infty)}f( r)dr\bigg)\,div_{\mu}( \varphi\,\psi)(x)\,d\mu 
 \\
 \\
 & = &  \ds -\int_{\R} f( r)\bigg(\int_X \one_{G^{-1}(r, +\infty)} div_{\mu}(\varphi\,\psi)d\mu\bigg)dr = -\int_{\alpha}^{\beta} dr \int_X \one_{G^{-1}(r, +\infty)} div_{\mu}(\varphi\,\psi)d\mu
 \\
 \\
 & = &  \ds \int_{\alpha}^{\beta} dr \int_X \one_{G^{-1}(-\infty, r)} (\varphi\,div_{\mu}\psi+\langle D_H\varphi,\psi\rangle_H)d\mu  
\end{array} $$
(in the last equality we used the fact that the divergence of any vector field in $W^{1,p}(X, \mu; H)$ has zero mean value).  Now, let $\varphi \in W^{1,p'}(X, \mu)$ and approach it by a sequence of smooth cylindrical  functions $\varphi_n$. After applying the above formula to each $\varphi_n$, we may let $n\to \infty$ in both sides, since $\varphi_n\to \varphi$ and  $\varphi_n\,div_{\mu}\psi+\langle D_H\varphi_n,\psi\rangle_H \to \varphi\,div_{\mu}\psi +\langle D_H\varphi,\psi\rangle_H$ in $L^1(X, \mu)$. Therefore we get 
$$ (\varphi\mu \circ G^{-1})((\alpha, \beta)) =   \int_{\alpha}^{\beta} dr \int_X \one_{G^{-1}(-\infty, r)} (\varphi\,div_{\mu}\psi+\langle D_H\varphi,\psi\rangle_H)d\mu ,$$
namely  $ \varphi\mu \circ G^{-1}$  has  density $q_\varphi$ given by
$$q_\varphi ( r)=\int_{G^{-1}(-\infty, r)}(\varphi\,div_{\mu}\psi+\langle D_H\varphi,\psi\rangle_H)\,\mu(dx),$$
which is continuous and bounded, since $\varphi\,div_{\mu}\psi+\langle D_H\varphi,\psi\rangle_H\in L^1(X, \mu)$ and $\mu(G^{-1}(r_0))=0$ for every $r_0\in \R$. Estimate \eqref{e5m} follows just applying the H\"older inequality. 
 \end{proof}


\subsection{Smoothness of densities}
\label{smoothness}


This \S \hspace{1mm} is devoted to show that for every uniformly continuous and bounded $\varphi : X\mapsto \R$, the function
\begin{equation}
\label{e1e}
F_{\varphi}(r ): = \int_{G^{-1}(-\infty, r)} \varphi \,d\mu
\end{equation}
is continuously differentiable.

A useful tool will be the following disintegration formula, whose proof is given for the reader's convenience in the appendix.

\begin{Theorem}
\label{tA.1}
Let  $X$ be  a Polish space endowed with a Borel probability measure $\mu$. Let $\Gamma :X \to \R$ be a Borel function, and set $\lambda = \mu\circ \Gamma ^{-1}$. 
Then there exists a family of  Borel probability measures $\{ m_s:\;s\in \R\}$ on $X$ such that
\begin{equation}
\label{eA.1}
\int_X \varphi(x)\mu(dx)=\int_\R \left(\int_{X} \varphi(x)m_s(dx)   \right)\lambda(ds),
\end{equation}
for all $\varphi: X\to \R$ bounded and Borel measurable.

Moreover the support of $m_s$
is contained  in $\Gamma ^{-1}(s)$ for $\lambda$-almost all  $s\in \R$.
\end{Theorem}

\begin{Proposition}
\label{C^1}
Let $\varphi\in BUC(X;\R)$. Then
$F_\varphi \in C^1_b(\R)$. 
\end{Proposition}
\begin{proof}
To begin with, let $\varphi:X\mapsto \R$ be Lipschitz continuous. By Lemma \ref{Nualartvarphi}, for each $r\in \R$ we have
$$F_\varphi( r)=\int_{-\infty }^rq_\varphi(s)ds, $$
where the function $q_\varphi\in L^1(\R)$ is continuous and bounded. Hence, $F_\varphi\in C^1_b(\R)$.

Taking $\Gamma =G$ and replacing $\varphi$   by $\varphi \one_{G^{-1}(-\infty, r)}$ we write the disintegration formula \eqref{eA.1}   as
\begin{equation}
\label{e8f}
F_\varphi(r )=\int_{-\infty}^r \left(\int_{X}\varphi(x)m_s(dx)  \right)q_1(s)ds,\quad r\in \R. 
\end{equation}
Therefore, there is a Borel set $I_\varphi\subset \R$ such that $(\mu\circ G^{-1})( I_\varphi)=0$ and
$$ F'_\varphi(r)=  q_1( r) \int_{X}\varphi(x)m_r(dx),\quad r\notin I_\varphi , $$
so that 
$$| F'_\varphi ( r)|\leq q_1( r)\|\varphi\|_{\infty} , \quad r \notin I_\varphi . $$
Since both $F'_\varphi$ and $q_1$ are continuous, 
\begin{equation}
\label{e9f}
| F'_\varphi ( r)|\leq q_1( r)\|\varphi\|_{\infty} , \quad r \in \R. 
\end{equation}
Let now $\varphi\in BUC(X;\R)$. By \cite[Thm.1]{Mic}, there is a sequence of Lipschitz continuous and bounded functions $\varphi_n$ that converge uniformly to $\varphi$ on $X$. Recalling that $| F_\varphi ( r)|\leq \|\varphi\|_{L^1(X, \mu)}\leq \|\varphi\|_{\infty}$ for every $r\in \R$,  estimate \eqref{e9f} yields that $(F_{\varphi_n})$ is a Cauchy sequence in $C^1_b(\R)$, and the conclusion follows. 
 \end{proof} 

For every $\varphi\in BUC(X;\R)$ we still set 
\begin{equation}
\label{q}
q_{\varphi}( r) := F_{\varphi}'( r), \quad r\in \R. 
\end{equation}
Of course, $q_{\varphi}$ is given by \eqref{e4m}  only if $\varphi\in W^{1,p'}(X, \mu)$.

\subsection{Surface measures}
\label{Surface}

Now we are ready to prove the existence of measures on {\em every} level surface $G^{-1}(r )$.

\begin{Theorem}
\label{costruzione}
For every $r\in \R$ there exists a unique   Borel measure $\sigma_r^G$ on ${\mathcal B}(X)$ such that 
\begin{equation}
\label{e15d}
 q_\varphi ( r)  =\int_{X} \varphi(x)\,\sigma_r^G(dx),\quad \varphi\in BUC(X; \R). 
\end{equation}
Moreover, the support of $\sigma_r^G$ is contained in $G^{-1}( r)$, and $\sigma_r^G(X)=q_1(r )$. Therefore, $\sigma_r^G$ is nontrivial  iff $q_1(r )> 0$.  
\end{Theorem} 
 \begin{proof}
Fix $r\in \R$ and set
\begin{equation}
\label{L}
L(\varphi):=  q_\varphi ( r)  =  F'_{\varphi}(r )  ,\quad  \varphi\in BUC(X; \R) \cup W^{1,p'}(X, \mu). 
\end{equation}
Since $F_{\varphi}$  is an increasing function for every $\varphi$ with nonnegative values, then $L(\varphi )   \geq 0$ if $\varphi(x)\geq 0$ a.e. 
Linear positive functionals  defined on $BUC(X; \R)$ have not necessarily an integral representation such as \eqref{e15d}. To show that this is the case, we use the following procedure. We approach $X$ by a sequence of compact sets $K_n$ such that $\lim_{n\to \infty} \mu(K_n) =1$ and we consider suitable restrictions $L_n$ of $L$ to $C(K_n;\R)$. By  the Riesz Theorem, such restrictions are represented by measures defined on the Borel sets of $K_n$,  readily extended to measures $\lambda_n$ on all Borel sets of $X$. Since  $(\lambda_n(B))$ is an increasing sequence for every Borel set $B$, we set  $\sigma_r^G(B) := \lim_{n\to \infty} \lambda_n(B)$ and  we prove that $\sigma_r^G$ is a  measure, that satisfies 
 \eqref{e15d}. 

Let $K$ be a compact subset of $X$ with positive measure. Since the embedding $H\subset X$ is compact, we may assume that $K $ contains $B_H(0,1)$. Moreover, replacing $K$ by its absolutely convex hull, we may assume that $K$ is symmetric (namely, $K=-K$) and convex. The linear span $E$ of $K$ is a measurable subspace of $X$ with positive measure; by the $0-1$ law (e.g., \cite[Thm. 2.5.5]{Boga}) it has measure $1$. Therefore, setting
$$K_n := nK, $$
we have
$$\lim_{n\to \infty} \mu(K_n) =1. $$
Now we follow a classical procedure in measure theory, see e.g. \cite[Ch. 6]{Beppe}. For any $n\in \N$  we consider the restriction $L_n$ of $L$ to $K_n$ defined for all $\varphi\ge 0$ as
$$L_n(\varphi)=\inf\{L(\psi):\;\psi\in BUC(X; \R),\;\psi= \varphi\;\mbox{\rm on}\;K_n,\;\;\psi\ge 0\;\mbox{\rm on}\;X \},$$
while if $\varphi $ takes  both positive and negative values, $L_n\varphi $ is defined by
$$L_n\varphi = L_n \varphi^+ - L_n \varphi^-, $$
where $\varphi^+ $ and $ \varphi^-$ denote the positive and the negative part of $\varphi$. The set used in the definition of $L_n$ is not empty, for instance it contains  the extension  studied in  \cite{Man90}, 
$$\psi (x) =\left\{ \begin{array}{ll}
\varphi(x), & x\in K_n, 
\\
\\
\inf_{u\in K_n}\frac{\varphi(u)}{\|x-u\|}{{\rm dist}\,(x, K_n)}, & x \notin K_n. 
\end{array}\right. $$
Then, $L_n$ is a positive linear functional in $C(K_n; \R)$. Positivity follows immediately from the positivity of $L$, linearity is not immediate although elementary, it may be proved as in Lemma 6.4 of \cite{Beppe}. Then, 
there exists a Borel measure  $\lambda_n$ on $K_n$ such that
$$ L_n(\varphi)=\int_{K_n} \varphi \,d\lambda_n,\quad \varphi\in  C(K_n; \R). $$
The obvious  extension of $\lambda_n$ to  ${\mathcal B}( X)$,  $B\mapsto  \lambda_n(B\cap K_n)$, is still denoted by $\lambda_n$. 

For every $\varphi\in BUC(X; \R)$ with nonnegative values the sequence $( L_n(\varphi))$ is increasing, since $\{L(\psi):\;\psi\in BUC(X; \R),\;\psi= \varphi\;\mbox{\rm on}\;K_{n+1},\;\;\psi\ge 0\;\mbox{\rm on}\;X \}$ $\subset $ $\{L(\psi):\;\psi\in BUC(X; \R),\;\psi= \varphi\;\mbox{\rm on}\;K_n,\;\;\psi\ge 0\;\mbox{\rm on}\;X \}$ for every $n\in \N$. 
It follows that for every  $B\in {\mathcal B}( X)$ the sequence $(\lambda_n(B))$ is increasing. Setting
$$\sigma_r^G (B) :=    \lim_{n\to \infty} \lambda_n(B), $$
we claim that $\sigma_r^G$ is a measure on ${\mathcal B}( X)$ and that \eqref{e15d} holds. 

Note that if $A$, $B$ are Borel sets such that $A\subset B$, then  $\sigma_n(A)\leq \sigma_n(B)$ for every $n$,  and consequently 
$\sigma_r^G(A)\leq  \sigma_r^G(B)$. 
Let now $B$, $B_m\in {\mathcal B}( X)$ be such that $B_m\uparrow B$. Then
$$\begin{array}{l}
\lim_{m\to \infty}\sigma_r^G(B_m) =   \lim_{m\to \infty}( \lim_{n\to \infty} \lambda_n(B_m)) =   \sup_{m\in \N}( \sup_{n\in \N} \lambda_n(B_m))
\\
\\
=   \sup_{n\in \N} ( \sup_{m\in \N}\lambda_n(B_m)) =  \sup_{n\in \N} \lambda_n(B) = \sigma_r^G(B). 
\end{array}$$
So, $\sigma_r^G $ is a measure. 
As a next step, we prove that \eqref{e15d} holds for $\varphi\equiv 1$. 
To this aim we construct a sequence of $W^{1,p'}(X, \mu)$ functions $\theta_n$, such that $\theta_n\equiv 1$ on $K_n$ and $\theta_n\equiv 0$ outside $K_{2n}$. The starting point is the Minkowsky functional of $K$, 
$$  \m(x): =\inf\left\{ \lambda\geq 0 : x  \in \lambda K \right\}, \quad x\in E, $$
which is positively homogeneous, sub--additive, and $H$-Lipschitz since $K$ contains the unit ball of $H$. Indeed, for any   $h\in   H$ we have
\[
\frac{h}{|h|_H}\in B_H(0,1)\subset K,   
\]
that is,  $h \in |h|_H  K$, that implies $\m(h)\leq  |h|_H $.
As a consequence, for any $x\in E$, $h\in H$, 
\[
\m(x+h)     \leq \m(x) +\m(h) \leq \m(x)+ |h|_H, 
\]
and 
$$\m(x) = \m(x+h-h) \leq \m(x+h) + \m(-h) \leq \m(x+h) + |h|_H, $$
so that 
\[
| \m(x+h)-\m(x)|\leq |h|_H. 
\]
Since $\mu(E)=1$, then the null extension of $\m$ to the whole of $X$   is $H$-Lipschitz,  so that it belongs to $W^{1,q}(X, \mu)$ for every $q>1$ (e.g., \cite[Ex. 5.4.10]{Boga}). Now, let $\alpha\in C^{\infty}_c(\R)$ be such that $\alpha\equiv 1$ in $[0, 1]$, $\alpha \equiv 0$ in $[2, +\infty)$, $0\leq \alpha \leq 1$, and set 
$$\theta_n (x) \left\{ \begin{array}{l}
=\ds  \alpha(\m(x/n) ), \quad x\in E, 
\\
\\
= 0, \quad x\notin E. 
\end{array}\right. $$
Then $\theta_n\in W^{1,q}(X, \mu)$ for every $q>1$. Recalling that for every $x\in E$ we have $\m(x/n)\leq 1$ iff $x\in K_n$, we obtain 
$\theta_n\equiv 1$ in $K_n$, $\theta_n \equiv 0$ outside $K_{2n}$, and
$$\lim_{n\to \infty} \|\theta_n-1\|_{W^{1,p'}(X, \mu)} = 0.$$
Indeed, $\lim_{n\to \infty} \theta_n(x) =1$ for every $x\in E$ and $0\leq \theta_n(x)\leq 1$, so that $\lim_{n\to \infty} \theta_n =1$ in $L^{p'}(X, \mu)$. Moreover, 
$$D_H\theta_n(x) = \frac{1}{n}\alpha' (  \m(x/n) ) D_H\m(x/n), $$
so that $\lim_{n\to \infty}   D_H\theta_n =0$ in $L^{p'}(X, \mu; H)$. 

Then, 
$$\sigma_r^G(X) = \lim_{n\to \infty} \lambda_n(X) = \lim_{n\to \infty}\int_X 1\,d\lambda_n = \lim_{n\to \infty} L_{n}(1)
= \lim_{n\to \infty} L_{2n}(1). $$
On the other hand, for every $\psi\in BUC(X)$ such that $\psi \geq 1$ in $K_{2n}$, $\psi\geq 0$ in $X$, we have $\psi \geq \theta_n$ and therefore $L\psi \geq L(\theta_n)$, since $L(\psi)- L(\theta_n)$ is the derivative at $r$ of the increasing function  $\xi\mapsto \mu\{ x: \;\psi(x) - \theta_n(x) \leq \xi\}$. 
Taking the infimum, we get $L_{2n}(1) \geq L(\theta_n)$. Since $\theta_n$ goes to $1$ in $W^{1,p'}(X, \mu)$, by Lemma \ref{Nualartvarphi}  $L(\theta_n)$ goes to $L(1) =q_{1}( r)$ as $n\to \infty$. This shows that 
\begin{equation}
\label{L(1)}
\sigma_r^G(X)=q_{1}( r) = L(1). 
\end{equation}
Now we show that \eqref{e15d} holds for any $\varphi\in BUC(X; \R)$. It is sufficient to prove that it holds for every  $\varphi \in BUC(X; \R)$ with   values in $[0, 1]$. In this case, by definition, 
$$L\varphi \geq L_{n}(\varphi_{|K_n}) = \int_{X}\varphi \,d\lambda_n $$
where the right--hand side converges to $\int_X \varphi \, d\sigma_r^G $ as $n\to \infty$, since the sequence $(\lambda_n )$ weakly converges to $\sigma_r^G$. Therefore, 
$$L\varphi \geq  \int_{X}\varphi \,d\sigma_r^G $$
Now we remark that  $1-\varphi$ has positive values, and using \eqref{L(1)} and the above inequality we get
$$q_1(r ) -L \varphi = L(1 -\varphi) \geq  \int_{X}(1 -\varphi)\,d\sigma_r^G = q_{1}( r)-  \int_{X} \varphi\,d\sigma_r^G $$
so that 
$$L \varphi \leq  \int_{X} \varphi\,d\sigma_r^G , $$
and \eqref{e15d} follows. 

It remains to prove that $\sigma_r^G$ has support in $G^{-1}( r)$. To this aim, we remark that for every $\eps>0$ and $\varphi\in BUC(X; \R)$ with support contained in $G^{-1}(-\infty, r-\eps) \cup G^{-1}(  r+ \eps , +\infty)$, the function $  F_{\varphi}$ is constant in $(r-\eps, r+\eps)$, and therefore 
$F_{\varphi}'( r) =0$. By \eqref{e15d}, $\int_X \varphi \,d\sigma_r^G=0$. So, the support of $\sigma_r^G$ is contained in $\cap_{\eps >0}G^{-1}[r-\eps, r+\eps] = G^{-1}( r)$. 
 \end{proof}

 \begin{Remark}
Let $m_r$ be the probability measures given by the disintegration theorem \ref{tA.1}.
If $X^*$ is separable, then for a.e. $r\in \R$ such that  $q_1( r)>0$ we have
$$\sigma_r^G = q_1( r) m_r . $$
 \end{Remark}
 \begin{proof}
 Fix any   $\varphi\in BUC(X; \R)$. Applying formula \eqref{eA.1} to the function $\varphi \one_{G^{-1}(-\infty, r)}$ we obtain
 $$F_{ \varphi}(r ) = \int_{-\infty}^r \left(\int_{X} \varphi(x)m_s(dx)   \right)q_1(s)ds, \quad r\in \R. $$
On the other hand, by \eqref{e15d} we have
 $$F_{ \varphi}(r ) = \int_{-\infty}^r \left(\int_{X} \varphi(x)\sigma_s^G (dx)   \right) ds, \quad r\in \R. $$
 Therefore, there exists a negligible $I_{\varphi}\subset \R$ such that for every $s\notin I_{\varphi}$ we have
$$\int_X \varphi \,d\sigma_s^G  = q_1(s )\int_{X} \varphi(x)\, m_s(dx)  . $$
Let ${\mathcal F } = \{f_n:\;n\in \N\}$ be any dense subset of $X^*$, and set $I = \cup_{n\in \N}I_{e^{if_n}}$. For every $r\notin I$ we have
 $$\int_X e^{if_n}\,d\sigma_r^G = q_1(r )\int_{X} e^{if_n(x)}\, m_r(dx). $$
Approaching  every $f\in X^*$ by a  sequence of elements of ${\mathcal F }$, and using the
Dominated Convergence Theorem, we obtain that if $q_1( r)\neq 0$, then the probability measures $\sigma_r^G / q_1(r )$ and $m_r$ have the same Fourier transform, so that they coincide. 
 \end{proof}

The following proposition shows a class of sets that are negligible with respect to all measures $\sigma_r^G$. 

\begin{Proposition}
\label{capacity}
Let $B\subset X$ be a Borel set with $C_{p'}(B) = 0$. Then $\sigma_r^G(B)=0$, for every $r\in \R$. 
\end{Proposition} 
\begin{proof}
We partly follow the argument used in \cite[Lemma 6.10.1]{Boga}. 
For every $\eps >0$ let  $O_{\eps}\supset B$ be an open set such that $C_{p'}(O_{\eps})<\eps$. Then there exists $f_{\eps}\in W^{1,p'}(X, \mu)$ such that 
$\|f_{\eps}\|_{W^{1,p'}(X, \mu)}\leq \eps$ and $f_{\eps}\geq 1$ a.e.  in $O_{\eps}$.  Replacing $f_{\eps}$ by $\max \{ f_{\eps}, 0\}$ we may assume that 
$f_{\eps}\geq \one_{O_{\eps}}$, $\mu$-a.e.

Let us fix a sequence of $BUC$  functions that converge to $\one_{O_{\eps}}$ pointwise. For instance, we can take
$$\theta_n(x) =\left\{ \begin{array}{ll}
0, & x\in X\setminus O_{\eps}, 
\\
\\
n\,{\rm dist}(x, X\setminus O_{\eps}), & 0<{\rm dist}(x, X\setminus O_{\eps})< 1/n, 
\\
\\
1, & {\rm dist}(x, X\setminus O_{\eps})\geq 1/n
\end{array}\right. $$
Then, $\lim_{n\to \infty} \theta_n(x)= \one_{O_{\eps}}(x)$, for every $x\in X$. Using  the Dominated Convergence Theorem, and then formula \eqref{e15d}, we get 
$$\sigma_r^G(O_{\eps}) = \int_{X} \one_{O_{\eps}} \,d\sigma_r^G = \lim_{n\to \infty} \int_{X}  \theta_n \,d\sigma_r^G  =  \lim_{n\to \infty} q_{ \theta_n}( r). $$
On the other hand, $f_{\eps}(x) \geq \one_{O_{\eps}}(x) \geq  \theta_n (x) $, for $\mu$-a.e. $x\in X$, so that the function
$F_{f_{\eps}- \theta_n }$ is increasing. In particular, $F_{f_{\eps}- \theta_n }'( r) =  q_{f_{\eps}}( r) - q_{ \theta_n}( r) \geq 0$ for every $r\in \R$. 
Therefore, for every $r\in \R$, 
$$\sigma_r^G(O_{\eps}) \leq  q_{f_{\eps}} ( r ). $$
On the other hand, by  \eqref{e5m} we have
$$|q_{f_{\eps}} ( r )| \leq C\|f_{\eps}\|_{W^{1,p'}(X, \mu)}\leq C\eps ,$$
with $C$ independent of $\eps$. Therefore, $\sigma_r^G(O_{\eps}) \leq  C\eps $, which implies  $\sigma_r^G(B) =0$. 
\end{proof}

Proposition \ref{capacity} clarifies the dependence of the measures $\sigma_r^G$ on the version of $G$ that we have fixed. Two versions of $G$ that coincide outside a set with null $C_{p'}$ capacity give rise to the same measures $\sigma_r^G$.

\subsection{Integration by parts formulae.}

To start with, we establish an integration formula for $C^1_b$ functions that is a first step towards an integration by parts formula. The proof follows arguments from \cite{DPL,Tracce} (in fact, it is a rewriting of a part of \cite[Prop. 4.1]{Tracce} in our setting).

\begin{Proposition} 
\label{partirozza}
Let $k\in \N$ be such that either $D_kG \in W^{1,p'}(X, \mu)$ or $D_kG\in BUC(X, \R)$. Then for every $\varphi\in C^1_b(X, \R)$ and for every  $r\in \R$ we have
\begin{equation}
\label{parti_rozza}
\int_{G^{-1}(-\infty, r)} (D_k\varphi -  \hat{v}_k\varphi )\,d\mu = q_{\varphi D_kG}( r). 
\end{equation}
Moreover, \eqref{parti_rozza} holds also for every  $\varphi\in W^{1, q}(X, \mu)$ provided $D_kG \in W^{1,s}(X, \mu)$ and
\begin{equation}
\label{q,s}
\frac{1}{q} + \frac{1}{s} +\frac{1}{p} \leq 1. 
\end{equation}
\end{Proposition}
\begin{proof} Fix $\varphi\in C^1_b(X; \R)$. For $\eps >0$ we define a function $\theta_{\eps}$ by 
$$\theta_{\eps}(\xi) :=\left\{ \begin{array}{ll} 
1, & \xi\leq  -\eps, 
\\
-\frac{1}{ \eps} \xi, &  -\eps <\xi < 0 , 
\\
0, & \xi \geq 0 .
\end{array}\right.  $$
and we consider the function
$$ x\mapsto  \varphi(x)  \theta_{\eps}(G(x)-r), $$
which belongs to $W^{1,p'}(X, \mu)$, and its derivative along  $v_k$ is  $  \theta_{\eps}'(G(x))D_k G(x)\varphi(x) $ $+$ $  \theta_{\eps}(G(x))D_k \varphi(x)$. Applying the integration by parts formula \eqref{partiX} we get 
\begin{equation}
\label{prima}
\int_X (D_k\varphi - \hat{v}_k \varphi)  (\theta_{\eps}\circ G)  \,d\mu = \frac{1}{ \eps} \int_{G^{-1}(r-\eps, r)} \varphi D_k G \,d\mu  , \quad k\in \N.
\end{equation}
As $\eps\to 0$, $\theta_{\eps}\circ G$ converges pointwise to  $ \one_{G^{-1}(-\infty, r)}$. Since  $0\leq  \theta_{\eps}\circ G \leq 1$, by the Dominated Convergence Theorem the left hand side converges to 
$$\ \int_{G^{-1}(-\infty, r)} (D_k\varphi -  \hat{v}_k\varphi) \,d\mu .$$
Concerning the right hand side, for every $\eps >0$ we have 
$$  \frac{1}{ \eps} \int_{G^{-1}(r-\eps, r)}    \varphi D_k G \,d\mu  =    \frac{1}{ \eps} \int_{r-\eps}^{r} q_{ \varphi D_k G}(\xi)d\xi . $$
Since $ \varphi D_k G$ belongs to  $W^{1,p'}(X, \mu)$ or to $ BUC(X;\R)$, by Lemma \ref{Nualartvarphi} or by Proposition \ref{C^1} the function $q_{ \varphi D_k G}$ is continuous. Therefore, 
$$\lim_{\eps \to 0}  \frac{1}{ \eps} \int_{G^{-1}(r-\eps, r)}    \varphi D_k G \,d\mu  = q_{ \varphi D_k G}( r) , $$
and \eqref{parti_rozza} follows. 

Let now  $\varphi\in W^{1,q}(X, \mu)$, $D_kG\in W^{1,s}(X, \mu)$, with $q$, $s$ satisfying \eqref{q,s}.   Let $\varphi_n\in C^1_b(X, \mu)$ approach 
$\varphi$ in $W^{1,q}(X, \mu)$, so that $\varphi_n D_kG$ approaches $\varphi D_kG$ in $W^{1,p'}(X, \mu)$. By \eqref{parti_rozza}, for every $r\in \R$ and $n\in \N$ we have
$$\int_{G^{-1}(-\infty, r)} (D_k\varphi_n -  \hat{v}_k\varphi _n)\,d\mu = q_{\varphi_n D_kG}( r). $$
Letting $n\to \infty$, the left hand side goes to $\int_{G^{-1}(-\infty, r)} (D_k\varphi_n -  \hat{v}_k\varphi _n)\,d\mu $, while the right hand side goes to $q_{\varphi D_kG}( r)$ by \eqref{e5m}. 
 \end{proof}

The measures $\sigma_r^G$ constructed in Theorem \ref{costruzione} are trivial  if $q_1 (r )=0$. 
So,  it is important to know for which values of $r$ we have $q_1(r )> 0$. 
This question was addressed in the paper \cite{HS}, where it was proved that under the assumptions of \cite{AM}, the set $I:= \{r\in \R:\;q_1( r)>0\}$ is an interval. Here we improve such a result, characterizing $I$ under more general assumptions and with a different simpler proof.

\begin{Lemma}
\label{HS}
Assume that for every $k\in \N$, $D_kG\in W^{1, p'}(X, \mu)\cup BUC(X, \R)$. Then $\{r\in \R:\;q_1( r)>0\} = (\essinf G, \esssup G)$. 
\end{Lemma}
\begin{proof} By Theorem \ref{costruzione}, if   $q_1(r ) =0$ then  $\int_X \varphi \,d\sigma_r^G =0$ for every $\varphi \in BUC(X, \R)$, which implies $q_{\varphi}(r ) =0$ for every $\varphi\in BUC(X, \R)$. Approaching any $\varphi\in W^{1,p'}(X, \mu)$ by a sequence of $C^1_b$ functions $\varphi_n$, it follows that $q_{\varphi}(r ) =0$ for every $\varphi\in W^{1,p'}(X, \mu)$.  Taking $\varphi= D_k\psi D_kG$, with any cylindrical smooth $\psi$,  we have $q_{\varphi} (r )=0$, and formula \eqref{parti_rozza} yields 
$$  \int_{G^{-1}(-\infty, r)} (D_{kk}\psi  - \hat{v}_kD_k \psi )\,d\mu = 0. $$
Summing over $k$ and using \cite[Thm. 5.8.3, Rem. 5.8.7]{Boga} we obtain
\begin{equation}
\label{kernel}
 \int_{G^{-1}(-\infty, r)} L\psi  \,d\mu = 0, 
 \end{equation}
where $L$ is the realization of the Ornstein--Uhlenbeck operator  in $L^2(X, \mu)$. We recall (e.g., \cite[Thm. 5.7.1]{Boga}) that the domain of $L$ is $W^{2,2}(X, \mu)$, and the graph norm of $L$ is equivalent to the $W^{2,2}$-norm. This implies that the set of the
 cylindrical smooth functions is a core for $L$, and then \eqref{kernel} holds for every $\psi\in D(L)$. In other words, the characteristic function $\one_{G^{-1}(-\infty, r)}$ is orthogonal to the range of $L$. Since $0$ is an isolated simple eigenvalue of $L$, the orthogonal space to the range of $L$ consists of constant a.e. functions. Then, $\one_{G^{-1}(-\infty, r)}$ is constant $\mu$-a.e., which implies that either $\mu(G^{-1}(-\infty, r)) =0$ or $\mu(G^{-1}(-\infty, r)) =1$. 
So, $ q_1( r)= 0$ implies that $r\in (-\infty, \essinf G] \cup [\esssup G, +\infty)$. 

Conversely, the function $F_1$ is continuously differentiable, and it is constant in $(-\infty, \essinf G]$ and in $[\esssup G, +\infty)$, so that for every $r\in (-\infty, \essinf G] \cup [\esssup G, +\infty)$ we have $F_1'(r ) =q_1(r ) =0 $. 
\end{proof}

Let us go back to Proposition \ref{partirozza}. 
We recall that $q_{\varphi}(r) = \int_X \varphi\,d\sigma_r^G$ if $\varphi\in BUC(X; \R)$. Therefore, if $G$ and $\varphi$ are so smooth that $\varphi D_kG\in  BUC(X; \R)$, \eqref{parti_rozza} yields
\begin{equation}
\label{parti_noi}
\int_{G^{-1}(-\infty, r)} (D_k\varphi -  \hat{v}_k\varphi )\,d\mu = \int_{G^{-1}( r)} \varphi D_kG\, d\sigma_r^G. 
\end{equation}
For more general $G$ and $\varphi$ the above formula still holds, but it is not obvious. For the right hand side of \eqref{parti_noi} to make sense, we need conditions guaranteeing that  $\varphi D_kG$ has a trace at $G^{-1}( r)$, belonging to $L^1(X ,\sigma_r^G)$. Then, $\varphi D_kG$ in the right hand side integral should be interpreted in the sense of traces. 

The starting point is Lemma \ref{Nualartvarphi} and in particular formula \eqref{e5m}, applied to the function $|\varphi|$, that together with \eqref{e15d}
yields
\begin{equation}
\label{stimatracce}
\int_{G^{-1}( r)}|\varphi|\,d\sigma_r^G = q_{|\varphi|}( r)  \leq C \|\varphi\|_{W^{1,p'}(X, \mu)}, \quad \varphi\in C^1_b(X; \R). 
\end{equation}
Since  $C^1_b(X; \R)$ is dense in $W^{1,p'}(X, \mu)$, the above estimate is extended to the whole of $W^{1,p'}(X, \mu)$, and it allows to define the traces of such Sobolev functions at $G^{-1}( r)$. Indeed, approaching any $\varphi\in W^{1,p'}(X, \mu)$ by a sequence of $C^1_b$ functions $\varphi_n$, \eqref{stimatracce} implies that the sequence of the restrictions $\varphi_{n|G^{-1}( r)}$ to $G^{-1}( r)$ is a Cauchy sequence in $L^1(G^{-1}( r), \sigma_r^G)$, that converges to an element of $L^1(G^{-1}( r), \sigma_r^G)$. Still by \eqref{stimatracce}, such element does not depend on the approximating sequence. 

\begin{Definition}
Let $\varphi \in W^{1,p'}(X, \mu)$. The trace of $\varphi$ at $G^{-1}( r)$ is the limit in $L^1(G^{-1}( r), \sigma_r^G)$ of the sequence of the restrictions $\varphi_{n|G^{-1}( r)}$ to $G^{-1}( r)$, for every sequence of $C^1_b$ functions $\varphi_n$ that converges to $\varphi $ in $ W^{1,p'}(X, \mu)$. It is denoted by $\varphi_{|G^{-1}( r)}$. 
\end{Definition}

By definition, $\varphi_{|G^{-1}( r)} \in L^1(G^{-1}( r), \sigma_r^G)$, and $\| \varphi_{|G^{-1}( r)} \|_{L^1 (G^{-1}( r), \sigma_r^G)}\leq C \|\varphi\|_{W^{1,p'}(X, \mu)}$, where $C$ is the constant in  \eqref{stimatracce}.  In other words, the trace is a bounded operator from $ W^{1,p'}(X, \mu)$ to $L^1(G^{-1}( r), \sigma_r^G)$. 
If $\varphi \in W^{1,q}(X, \mu)$, with $q>p'$, then $|\varphi |^{q/p'}\in W^{1,p'}(X, \mu)$, and estimate \eqref{stimatracce} applied to $|\varphi |^{q/p'}$ yields that the trace of $\varphi$ at $G^{-1}( r)$  belongs to $L^{q/p'}(G^{-1}( r), \sigma_r^G)$, and the trace operator is bounded from $W^{1,q}(X, \mu)$ to $L^{q/p'}(G^{-1}( r), \sigma_r^G)$. 

The trace operator preserves positivity, as the next lemma shows.

\begin{Lemma}
\label{Le:positivita'}
Let $\varphi \in W^{1,p'}(X, \mu)$ have nonnegative values, $\mu$-a.e. Then for every $r\in \R$ the trace of $\varphi$ at $G^{-1}( r)$ has nonnegative values, $\sigma_r^G$-a.e.
\end{Lemma}
\begin{proof} 
Let $(\varphi_n)$ be a sequence of $C^1_b$ functions, converging to $\varphi$ in $W^{1,p'}(X, \mu)$. 
Possibly replacing $(\varphi_n)$ by a subsequence, we may assume that $(\varphi_n)$ converges to $\varphi $ pointwise $\mu$-a.e.

We claim that the sequence $(\varphi_n^+)$ (the positive parts of $\varphi_n$) still converges to $\varphi$ in $W^{1,p'}(X, \mu)$. Indeed, $\| \varphi_n^+-\varphi\|_{L^{p'}(X, \mu)} \leq \| \varphi_n -\varphi\|_{L^{p'}(X, \mu)}$, while, recalling that $D_H\varphi_n^+ = D_H\varphi_n$ in the set $\{x:\; \varphi_n(x)> 0\}$, and $D_H\varphi_n^+ = 0$ in the set $\{x:\; \varphi_n(x) \leq   0 \}$, $D_H\varphi  = 0$ in the set $\{x:\; \varphi(x) =  0 \}$  (\cite[Lemma 5.7.7]{Boga}) we obtain
\begin{equation}
\label{stima}
\begin{array}{l}
\ds \int_{X} |D_H\varphi_n^+-D_H\varphi|^{p'}_{H} d\mu = \int_{\{x:\; \varphi_n(x)> 0\}} |D_H\varphi_n^+-D_H\varphi|^{p'}_{H} d\mu 
+ \int_{\{x:\; \varphi_n(x) \leq 0\}} |D_H\varphi_n^+-D_H\varphi|^{p'}_{H} d\mu 
\\
\\
= \ds \int_{\{x:\; \varphi_n(x)> 0\}} |D_H\varphi_n -D_H\varphi|^{p'}_{H} d\mu 
+ \int_{\{x:\; \varphi_n(x) \leq  0 \}} | D_H\varphi|^{p'}_{H} d\mu 
\\
\\
\leq \ds \| D_H\varphi_n -D_H\varphi\|_{L^{p'}(X, \mu;H)}^{p'} +  \int_{\{x:\; \varphi_n(x) \leq  0, \,\varphi (x)>0 \}} | D_H\varphi|^{p'}_{H} d\mu 
.
\end{array}\end{equation}
Setting $A_n: =\{ x:\; \varphi_n(x) \leq  0, \,\varphi (x)>0 \}$,  then $\mu (A_n)$ vanishes as $n\to \infty$. This is because for every $x$ in the set 
$$A := \bigcap_{n=1}^{\infty} \bigcup_{k\geq n} A_k, $$
the sequence $(\varphi_n(x))$  does not converge to $\varphi(x)$, and therefore $0 = \mu(A) = \lim_{n\to \infty} \mu(\cup_{k\geq n} A_k)  \geq 
\limsup_{n\to \infty} \mu(A_n)$. 
Then, the right hand side of \eqref{stima} vanishes as $n\to \infty$, and this implies that $(\varphi_n)^+$ converges to $\varphi$ in $W^{1,p'}(X, \mu)$. Consequently, the traces of $(\varphi_n)^+$ at $G^{-1}( r)$ converge to the trace of $\varphi$ at $G^{-1}( r)$, in $L^1( G^{-1}( r), \sigma_r^G)$. 
Since each $(\varphi_n)^+$ has nonnegative values at every $x\in G^{-1}( r)$, then their $L^1$ limit has nonnegative values, $ \sigma_r^G$-a.e.
\end{proof}

Formula \eqref{e15d}  may now be extended to elements of $W^{1,p'}(X, \mu)$. 

\begin{Lemma}
\label{Le:identificazione}
For every $\varphi\in W^{1, p'}(X, \mu)$ and for every $r\in \R$ we have
\begin{equation}
\label{identificazione}
q_{\varphi} (r ) = \int_{G^{-1}( r)} \varphi_{|G^{-1}( r)}\, d\sigma_r^G .
\end{equation}
\end{Lemma} 
\begin{proof}
It is sufficient to approximate $\varphi$ in $ W^{1, p'}(X, \mu)$ by a sequence of functions $\varphi_n\in C^1_b(X; \R)$, and to let $n\to \infty $ in the equality
$$q_{\varphi_n} (r ) = \int_X  \varphi_n\, d\sigma_r^G , $$
that holds by Theorem \ref{costruzione}. The left hand side goes to $q_{\varphi} (r ) $ by estimate \eqref{e5m}, the right hand side goes to  $\int_{G^{-1}( r)} \varphi_{|G^{-1}( r)}\, d\sigma_r^G $ by the above construction of the trace of $\varphi$. 
\end{proof}

With the aid of Proposition \ref{capacity} we can prove that the traces of the elements of $W^{1,p'}(X, \mu)$ at $G^{-1}( r)$ coincide with the restrictions of their $C_{p'}$-quasicontinuous versions at $G^{-1}( r)$. In particular, if $\varphi$ is a continuous version of a Sobolev function, its trace is just the restriction of $\varphi$ at $G^{-1}( r)$. This justifies the notation $ \varphi_{|G^{-1}( r)}$ for the trace of $\varphi$ at $G^{-1}( r)$. 

\begin{Proposition}
\label{traccia=restrizione}
Let $\varphi $ be a $C_{p'}$-quasicontinuous version of an element of $W^{1,p'}(X, \mu)$. Then the trace of $\varphi $ at $G^{-1}( r)$ coincides with the restriction of $\varphi$  at $G^{-1}( r)$, $\sigma_r^G$-a.e.
\end{Proposition} 
\begin{proof} We use   arguments similar to  \cite[Prop. 4.8]{Tracce}. 
Let $(\varphi_n)$ be a sequence of smooth cylindrical functions that converge to $\varphi $  in $W^{1,p'}(X, \mu)$. By \cite[Thm. 5.9.6(ii)]{Boga}, applied with the operator $T= (I-L)^{-1/2}$, a subsequence $(\varphi_{n_k})$ converges to $\varphi(x)$  for every $x$ expect at most on a set with zero Gaussian capacity $C_{p'}$. By Proposition \ref{capacity}, such a subsequence converges $\sigma_r^G$-a.e to $\varphi$. On the other hand, by the definition of the trace,   the restrictions of $\varphi_n$ to $G^{-1}( r)$  converge to $\varphi_{|G^{-1}( r)} $ in $L^1(G^{-1}( r), \sigma_r^G)$. In particular, a subsequence of $(\varphi_{n_k})$ converges to $\varphi_{|G^{-1}( r)} $, $\sigma_r^G$-a.e. Therefore, $\varphi_{|G^{-1}( r)} = \varphi$, $\sigma_r^G$-a.e.
\end{proof}

To extend  the integration by parts formula \eqref{parti_noi} to Sobolev functions we need some further assumptions on $G$.

\begin{Corollary}
\label{parti_generale}
Let $\varphi\in W^{1, q}(X, \mu)$ and  $D_kG \in W^{1,s}(X, \mu)$, with $q$, $s$ satisfying \eqref{q,s}. 
Then formula 
\eqref{parti_noi} holds, with  $\varphi D_kG$   replaced by $( \varphi D_kG)_{| G^{-1}( r)}$. 
\end{Corollary}
\begin{proof}
Note that $\varphi D_kG \in W^{1,p'}(X, \mu)$.   By Lemma \ref{Le:identificazione} we have
$$\int_{G^{-1}( r)}( \varphi D_kG)_{| G^{-1}( r)} \,d\sigma_r^G = q_{\varphi D_kG}( r).  $$
On the other hand,   Proposition \ref{partirozza}  yields
$$q_{\varphi D_kG}( r)= \int_{G^{-1}(-\infty, r)} (D_k\varphi -  \hat{v}_k\varphi )\,d\mu, $$
and the statement follows. 
\end{proof}
  
Corollary \ref{parti_generale} yields an integration by parts formula for $\varphi\in W^{1,s}(X, \mu)$ for any $s>1$, provided $G$ is good enough. In particular, if next assumption \eqref{FP} holds, then $G$ satisfies \eqref{G} and the conditions of Corollary \ref{parti_generale} with any $s>1$, and \eqref{parti_noi} holds for $\varphi\in W^{1,q}(X, \mu)$ for any $q>1$.

\subsection{Dependence on $G$, and relationship with other surface measures}
\label{varie}

Now we are ready to compare the measures $\sigma_r^G$  defined in \S \ref{Surface} with the perimeter measure and with the Hausdorff-Gauss surface measure $\rho$ of Feyel and de La Pradelle \cite{FP}. 

We use the notation of \cite{AMMP}. We recall that a subset $B\subset X$ is said to have finite perimeter if $\one_B$ is a bounded variation function, namely there exists a $H$-valued measure $\Gamma$ such that for every $k\in \N$ and for every smooth cylindrical function $\varphi$ we have 
$$\int_B (D_k\varphi - \hat{v}_k\varphi)d\mu = \int_X \varphi \,d\gamma_k,  $$
with $\gamma_k = \langle \Gamma, v_k\rangle_H$. In this case, $\Gamma $ is unique,  it is called perimeter measure, and denoted by $D_{\mu} \one_B$. 

If   $G\in W^{2,p'}(X, \mu)$, 
then for every $r\in \R$ the set $B= G^{-1}(-\infty, r)$ satisfies the above condition, with $D_{\mu} \one_{G^{-1}(-\infty, r)} = D_HG_{|G^{-1}(r )} \,\sigma_r^G$. Indeed, by formulae \eqref{parti_rozza} (applied to $\varphi $) 
and \eqref{identificazione} (applied to $\varphi D_kG$),  for every for smooth cylindrical $\varphi$ and for every $k\in \N$ we have
$$\int_{G^{-1}(-\infty, r)} (D_k\varphi -  \hat{v}_k\varphi )\,d\mu = \int_{X} (\varphi D_kG)_{| G^{-1}( r)} \,d\sigma_r^G .$$
On the other hand, for smooth cylindrical $\varphi$ the trace of $\varphi D_kG$ at the support $G^{-1}( r)$ of $\sigma_r^G$ coincides $\sigma_r^G$-a.e. with the restriction of $\varphi \widetilde{D_kG}$ at $G^{-1}( r)$, where $\widetilde{D_kG}$ is any $C_{p'}$-quasicontinuous version of $D_kG$, 
by Proposition \ref{traccia=restrizione}.

Let us now recall the assumptions of Feyel \cite{F}, 

\begin{equation}
\label{FP}
G\in \bigcap_{p>1}W^{2, p}(X, \mu), \quad \frac{1}{|D_HG|_H} \in \bigcap_{p>1} L^p(X, \mu), 
\end{equation}
under which it was proved that for every  $\varphi \in W^{1,q}(X; \R)$ for some $q>1$, the density $q_{\varphi}$ of the signed measure $\varphi \mu\circ G^{-1}$  with respect to the Lebesgue measure is  given by
\begin{equation}
\label{Feyel}
q_{\varphi}( r) = \int_{G^{-1}(r )} \frac{\varphi}{|D_HG|_H}\,d\rho, \quad r\in \R. 
\end{equation}
In the right hand side, $\varphi$ and $|D_HG|_H$ are quasicontinuous versions of the respective Sobolev elements. More precisely, 
$\varphi$ is $C_q$-quasicontinuous and $|D_HG|_H$ is $C_p$-quasicontinuous for every $p>1$. 
See \cite{F,Tracce}.

\begin{Proposition}
Let $G$ satisfy \eqref{FP}. Then for every Borel set $B\subset X$ we have
$$\sigma_r^G (B) = \int_{B\cap G^{-1}( r)}\frac{1 }{|D_HG|_H}\,d\rho , \quad r\in \R. $$
\label{comparison}
\end{Proposition}
\begin{proof}
Note that if \eqref{FP} holds, then $G$ satisfies \eqref{G}. Comparing \eqref{Feyel} with \eqref{e15d} yields
$$\int_X \varphi\,d\sigma_r^G = \int_{G^{-1}(r )} \frac{\varphi}{|D_HG|_H}\,d\rho,  \quad \varphi\in C^1_b(X; \R), $$
and the statement holds.  \end{proof}

\begin{Corollary}
Let $G_1$, $G_2$ satisfy \eqref{FP}. 
Assume that for some $r_1$, $r_2\in \R$ we have $G_1^{-1}( r_1) = G_2^{-1}( r_2) := \Sigma$.  Then
$$|D_HG_1(r_1)|_H d\sigma_{r_1}^{G_1} = |D_HG_1(r_2)|_H d\sigma_{r_2}^{G_2} = \rho_{|\Sigma}. $$
\end{Corollary}

We recall that the assumptions of \cite{AM} are 
\begin{equation}
\label{AM}
G\in \bigcap_{k\in \N, \,p>1}W^{k, p}(X, \mu), \quad \frac{1}{|D_HG|_H} \in \bigcap_{p>1} L^p(X, \mu), 
\end{equation}
and that the measures $\nu_r$ constructed in \cite{AM,Boga} under such assumptions, for all $r\in \R$ such that $q_1(r )>0$, coincide with the restriction of 
$\rho$ to $G^{-1}(r )$. This is because they satisfy the same integration by parts formula,  
\begin{equation}
\label{parti_loro}
\int_{G^{-1}(-\infty, r)} (D_k\varphi -  \hat{v}_k\varphi )\,d\mu = \int_{G^{-1}( r)} \frac{\varphi D_kG}{|D_HG|_H}\, d\nu_r 
=  \int_{G^{-1}( r)} \frac{\varphi D_kG}{|D_HG|_H}\, d\rho,  \quad \varphi\in  \bigcap_{k\in \N, \,p>1}W^{k, p}(X, \mu), 
\end{equation}
and replacing $\varphi$ by $\varphi D_kG/ |D_HG|_H$ and summing up, we obtain 
$$\int_{G^{-1}( r)} \varphi  \, d\nu_r 
=  \int_{G^{-1}( r)} \varphi  \, d\rho,  \quad \varphi\in  \bigcap_{k\in \N, \,p>1}W^{k, p}(X, \mu)$$
which  implies the statement.

\vspace{3mm}

\appendix
\section{Proof of the disintegration theorem}

 We follow here \cite{LS}. For any $A\in \mathcal B(X)$ we  consider    the conditional expectation 
 $\E[\one_A|\Gamma] $, which may be expressed as $f_A\circ \Gamma$ for some Borel function $f_A:\R \mapsto \R$. 
So,  for any $I\in \mathcal B(\R)$ we have
$$
\int_{\Gamma ^{-1}(I)}\one_A\,d\mu=\int_{\Gamma ^{-1}(I)}(f_A\circ \Gamma) \,d\mu=\int_I f_A( r)\, (\mu\circ \Gamma^{-1})(dr)
$$
which we rewrite as 
\begin{equation}
\label{eA.3}
\mu(A\cap \Gamma^{-1}(I)) = \int_I f_A( r) \lambda (dr).
\end{equation}
Since $X$ is separable,  there exists  $K\subset\R$ with $\lambda(K)=0$ and  for any $r\notin K$ a Borel measure $m_r$ on $\R$ such that
$$f_A(r)=m_r(A),\quad\forall\;r\notin K.$$
See e.g. \cite[Theorem 10.2.2]{D}. 
Replacing in  \eqref{eA.3} we obtain
 \begin{equation}
\label{eA.4}
\mu(A\cap \Gamma ^{-1}(I)) =\int_I m_r(A)\lambda (dr),\quad I\in\mathcal B(\R). 
\end{equation}

It is enough to prove that  \eqref{eA.1} holds for  $\varphi=\one_{\Gamma^{-1}(J)}$  with $J\in \mathcal B(X)$.  In this case, 
 $$
 \int_X\varphi(x)m_r(dx)=m_r(\Gamma^{-1}(J))
 $$
and integrating with respect to $\lambda$ over $\R$ and taking into account of \eqref{eA.4} with $I=\R$, yields
$$\int_\R \left(\int_{X} \varphi(x)m_r(dx)   \right)\lambda (dr)=\int_\R m_r( \Gamma^{-1}(J))\lambda (dr)=\mu(\Gamma^{-1}(J)),
$$
as claimed.\medskip

Let us show that for $\lambda$-a.e. $r_0\in \R$, the support of $m_{r_0}$ is contained in $\Gamma^{-1}(r_0)$. If $I$ is any interval, setting  $A=\Gamma^{-1}(\R \setminus I)$ in \eqref{eA.4}, 
we find
$$0 =\int_{I} m_r (\Gamma^{-1}(\R \setminus I)) \lambda (dr), $$
so that  $m_r (\Gamma^{-1}(\R \setminus I))=0$ for $\lambda$-almost all $r\in I$, say for every $r\in I\setminus J_I$ with $\lambda (J_I)=0$. Now, let us consider all the open intervals with rational endpoints, $I=(a_n, b_n)$ with $a_n<b_n\in \Q$. The set $J := \cup_{n\in \N} J_{(a_n, b_n)}$  is still $\lambda$-negligible, and   we have 
\begin{equation}
\label{eA.5}
m_r (\Gamma^{-1}((-\infty, a_n] \cup [b_n, +\infty) ) )= 0, \quad n\in \N, \;r\in (a_n, b_n)\setminus J. 
\end{equation}
For every $r_0\in \R\setminus J$, fix two subsequences $(a_{n_k})$, $(b_{n_k})$ such that $a_{n_k}< r_0$, $b_{n_k}>r_0$ and $\lim_{k\to \infty}
a_{n_k} =\lim_{k\to \infty}b_{n_k} =r_0$. Taking $r=r_0$ and replacing $a_n$, $b_n$ by $a_{n_k}$, $b_{n_k}$ in  \eqref{eA.5}, we obtain that $m_{r_0}$ has support contained in $(a_{n_k}, b_{n_k})$ for every $k\in \N$, and the statement follows.


\begin{thebibliography}{99}

\bibitem{AM}
 {\sc H. Airault,  P. Malliavin}:  \textit{Int\'egration g\'eom\'etrique sur l'espace de Wiener},  Bull. Sci. Math. {\bf 112},  3--52  (1988). 

\bibitem{AMMP} {\sc 
L. Ambrosio,  M. Miranda Jr., S. Maniglia,  D. Pallara}:
\textit{BV functions in abstract Wiener spaces}, J. Funct. Anal. {\bf 258} (2010), 785--813. 
 
\bibitem{Boga} {\sc V.I. Bogachev}: \textit{Gaussian Measures}, American Mathematical Society, Providence (1998).

\bibitem{CLMN} {\sc V. Caselles, A. Lunardi, M. Miranda Jr, M. Novaga:} 
{\it Perimeter of sublevel sets in infinite dimensional spaces}, Adv. Calc. Var. {\bf 5} (2012), 59--76.

\bibitem{Tracce} {\sc P. Celada, A. Lunardi}:  {\it Traces of Sobolev functions on regular surfaces in infinite dimensions}, J. Funct.  Anal. {\bf 266} (2014), 1948--1987. 

\bibitem{Beppe} {\sc G. Da Prato}: {\it An introduction to infinite dimensional analysis}, Springer-Verlag, Berlin (2006). 
 
 \bibitem{DPL} {\sc G. Da Prato, A. Lunardi}:  {\it On the Dirichlet semigroup for Ornstein--Uhlenbeck 
operators in subsets of Hilbert spaces}, J. Funct. Anal. {\bf 259} (2010), 2642--2672. 
 

\bibitem{D}  {\sc  R.M. Dudley}: \textit{Real Analysis and Probability}, Cambridge Univ. Press, Cambridge (2002). 

\bibitem{F}   {\sc  D. Feyel}:  \textit{Hausdorff-Gauss Measures}, in: Stochastic Analysis and Related Topics, VII. Kusadasi 1998, Progr. in Probab. 98, Birkh\"auser (2001), 59--76. 
 
\bibitem{FP}   {\sc   D. Feyel,  A. de La Pradelle}: \textit{ Hausdorff measures on the Wiener space}, Pot. Analysis {\bf 1} (1992), 177--189. 
 
 \bibitem{FH}  {\sc  M. Fukushima, M. Hino}: \textit{On the space of BV functions and a related stochastic calculus in infinite dimensions},
J. Funct. Anal. {\bf 183} (2001), 245-268.
 
\bibitem{Hertle}  {\sc  A. Hertle}:  \textit {Gaussian surface measures and the Radon transform on separable Banach
spaces}, Lecture Notes in Math. 794, 513--531. Springer-Verlag (1980).

\bibitem{Hino}  {\sc  M. Hino}: \textit{Sets of finite perimeter and the Hausdorff--Gauss measure on the Wiener space}, J. Funct. Anal. {\bf 258} (2010), 1656--1681.

\bibitem{HS} {\sc F. Hirsch, S. Song}: \textit{Properties of the set of positivity for the density of a regular Wiener functional}, 
 Bull. Sci. Math. {\bf 122} (1998), 1--15. 

\bibitem{Kuo} {\sc S. Chaari, F. Cipriano,  H.-H. Kuo, H. Ouerdiane}:  \textit{Surface measures on the dual space of the Schwartz space}, Comm. Stoch. Anal. {\bf 4} (2010), 467--480. 
 
\bibitem{Malliavin}  {\sc  P. Malliavin}:   \textit{Stochastic analysis},   Springer-Verlag, Berlin (1997).

\bibitem{Man90} {\sc M. Mandelkern}: \textit{On the uniform continuity of Tietze extensions}, Arch. Math. {\bf 55} (1990), 387--388. 

\bibitem{Mic}  {\sc  R. Miculescu}: \textit{Approximation by Lipschitz functions generated by extensions}, Real Anal. Exchange {\bf 28} (2002/03), 33--40. 

\bibitem{Nualart}   {\sc  D. Nualart}:  \textit{The Malliavin calculus and related topics}. Probability and its Applications, Springer-Verlag,  Berlin (1995).

\bibitem{LS} {\sc L. Schwartz}: \textit{Surmartingales r\'eguli\`eres \`a valeurs mesures et d\'esint\'egration r\'egulie\`re d'une mesure}, J. Analyse Math. {\bf 26} (1973), 1--168. 

\bibitem{Sk74} {\sc  A.V. Skorohod}: \textit{Integration in Hilbert Space},  Springer Verlag, New York (1974). 

\bibitem{Ug}  {\sc  A.V. Uglanov}: \textit{Surface integrals in a Banach space}. Matem. Sb. {\bf 110} (1979), 189--217;
English transl.: Math. USSR Sb. {\bf 38} (1981), 175--199.

\bibitem{Ug2}   {\sc  A.V. Uglanov}: \textit{Integration on
Infinite-Dimensional Surfaces and Its Applications}. Springer Science+Business Media, B.V. (2000). 
 
  
\end{thebibliography}
\end{document}